\documentclass[10pt, reqno]{amsart}

\usepackage{amsmath,amssymb,amsthm,amsfonts,verbatim}
\usepackage{microtype}
\usepackage[all,2cell]{xy}
\usepackage{mathtools}
\usepackage{graphicx}
\usepackage{pinlabel}
\usepackage{hyperref}
\usepackage{mathrsfs}
\usepackage{color}
\usepackage{enumerate}
\usepackage{cite}
\usepackage{tcolorbox}

\usepackage[top=1.3in,bottom=1.9in,left=1.2in,right=1.2in]{geometry}

\theoremstyle{plain}
\newtheorem{thm}{Theorem}[section]

\newtheorem{lem}[thm]{Lemma}

\newtheorem{prop}[thm]{Proposition}
\newtheorem{cor}[thm]{Corollary}

\newtheorem{conj}[thm]{Conjecture}
\newtheorem{qu}[thm]{Problem}
\theoremstyle{definition}
\newtheorem{defn}[thm]{Definition}

\theoremstyle{remark}
\newtheorem*{rem}{Remark}

\newcommand{\nc}{\newcommand}
\nc{\dmo}{\DeclareMathOperator}

\DeclareMathOperator{\Conf}{Conf}

\DeclareMathOperator{\Fix}{Fix}
\DeclareMathOperator{\Homeo}{Homeo}

\DeclareMathOperator{\ZZ}{\mathbb{Z}}
\DeclareMathOperator{\FF}{\mathbb{F}}

\renewcommand{\epsilon}{\varepsilon}

\nc{\coloneq}{\mathrel{\mathop:}\mkern-1.2mu=}
\nc{\margin}[1]{\marginpar{\scriptsize #1}}
\nc{\para}[1]{\medskip\noindent\textbf{#1.}}
\nc{\red}[1]{\textcolor{red}{#1}}

\title{Actions of homeomorphism groups of manifolds admitting a nontrivial finite free action}
\address{\newline Department of Mathematics   \newline California Institute of Technology   \newline Pasadena, CA 91125,  USA}
\author{Lei Chen}
\date{September 26th}
\email{L.C. chenlei@caltech.edu}
\begin{document}
\maketitle
\begin{abstract}
In this paper, we study the action of $\Homeo_0(M)$, the identity component of the group of homeomorphisms of an $n$-dimensional manifold $M$ with an $\FF_p$-free action, on another manifold $N$ of dimension $n+k<2n$. We prove that if $M$ is not an $\FF_p$-homology sphere, then $N\cong M\times K$ for a homology manifold $K$ such that the action of $\Homeo_0(M)$ on $M$ is standard and on $K$ is trivial. In particular, for $M=S^n$ a sphere, any nontrivial action is a generalization of the ``coning-off" construction.
\end{abstract}
\section{Introduction}
For a manifold $M$, denote by $\Homeo(M)$ the group of homeomorphisms of $M$ and by $\Homeo_0(M)$ the identity component of the group of homeomorphisms of $M$.  Let $M,N$ be two connected, closed, topological manifolds, and let \[
\rho: \Homeo_0(M)\to \Homeo(N)
\] 
be a group homomorphism. Recently the authors \cite{CM} showed that every orbit of the action $\rho$ is a ``submanifold" (continuous, injective image of a manifold) of $N$; in particular, every orbit is a continuous, injective image of a cover of $\Conf_m(M)$ for some $m$, where $\Conf_m(M)$ denotes the space of unordered $m$-tuples of distinct points on $M$. This orbit rigidity result means that any nontrivial action $\rho$ is, in some sense, geometric. It reduces the study of $\Homeo_0(M)$ acting on $N$ to the study of how to decompose $N$ into a union of covers of $\Conf_m(M)$ so that actions of $\Homeo_0(M)$ on pieces glue together. Do all possible decompositions geometric in some sense? We are puzzled by the following question: Given $M$, for which manifolds $N$, do there exist a nontrivial $\rho$? How many different actions can we find?

\para{Some results in low dimensions}
The followings have been proved in \cite{CM}.
\begin{itemize}
\item if $\dim(M)>\dim(N)$, then there is no nontrivial $\rho$; 
\item if $\dim(M)=\dim(N)$, the existence of a nontrivial $\rho$ is equivalent to the fact that an admissible cover (which is defined in \cite[Section 2.1]{CM}) of $M$ can be embedded in $N$; 
\item if $\dim(M)+1=\dim(N)$ and every admissible cover of $M$ is compact, then $\rho$ exists if and only if $M=S^n$. 
\end{itemize}
What will happen when $\dim(N)$ is bigger?

\para{A zoo of examples}
\begin{itemize}
\item \textbf{Product action:} $\Homeo_0(M)$ acts on $M^k\times K$ (product of $k$ copies of $M$ with $K$) where the action is diagonal on $M^k$ and trivial on $K$. 
\item \textbf{Product action quotient by Bing's shrinking:} Let $p: K\to K'$ be a Bing's shrinking map \cite[Chapter 2]{decompositions}. Define $M_p:=M\times M\times K/\sim$, where $(m,m,k_1)\sim (m,m,k_2)$ if $p(k_1)=p(k_2)$. The action of $\Homeo_0(M)$ on $M\times M\times K/\sim$ is the diagonal on $M\times M$ and trivial on $K$. The fact that $M_p$ is a topological manifold follows from the fact that $p$ is a Bing's shrinking map.
\item \textbf{The suspension action:} $\Homeo_0(S^n)$ acts on $S^{n+l}$ by lifting the standard action to the $l$-fold suspension $\Sigma^l S^n=S^{n+l}$.
\item \textbf{Double suspension action:} If $H$ is an $n$-dimensional homology manifold, then the double suspension theorem \cite{doublesuspension} implies that there is a homeomorphism $\Sigma^2H\cong S^{n+2}$. This means that $\Homeo_0(H)$ acts on $S^{n+2}$ nontrivially with global fixed point set a circle. 
\end{itemize}
This is a showcase that the action can be quite weird, we can multiply and quotient under Bing's shrinking map. Therefore we ask the following general question.
\begin{qu}
Are there constructions other than modifying geometric actions? 
\end{qu}
When the dimension $n+k$ of $N$ satisfies $k+n<2n$, we only have two types of orbit: either a global fixed point or a cover of $M$. The difficulty in this range is how to control the topology of the global fixed point set. In this paper, we make use of  the``shrinking property", which we discuss in Section 2 to control the local topology of the global fixed point set. We call $\tau\in \Homeo_0(M)$ a \emph{free order $p$ element} if $\tau$ has order $p$ and $\tau$ has no global fixed points on $M$. With the extra condition that there exists a free order $p$ element in $\Homeo_0(M)$, we deduce the following.
\begin{thm}\label{main}
Assume that admissible covers of $M$ are trivial and there exists a free order $p$ element in $\Homeo_0(M)$. Given a nontrivial homomorphism \[
\rho: \Homeo_0(M)\to \Homeo(N),
\]
then one of the following happens:
\begin{enumerate}
\item $N\cong M\times K$ for a $\ZZ$-homology manifold $K$ such that the action of $\Homeo_0(M)$ on $M$ is standard and on $K$ is trivial;
\item $M$ is an $\FF_p$-homology sphere and the global fixed point set is a union of finite $(k-1)$-dimensional $\FF_p$-homology manifolds.
\end{enumerate}
\end{thm}
\begin{rem}
\begin{itemize}
\item The above theorem applies for $M=G\times M'$ where $M'$ is simply connected and $G$ is a compact Lie group. 

\item We do not know if (2) of Theorem \ref{main} ever happens when $M$ is not a $\mathbb{Z}$-homology sphere. While if $M$ is a $\ZZ$-homology sphere, we know that such actions exist because of the double suspension theorem \cite{doublesuspension}.

\item We do not know if the above theorem depends on whether or not $\Homeo_0(M)$ contains a free order $p$ element. We are not able to construct any non-product example when $M$ is simply connected but not a sphere.
\end{itemize}
\end{rem}
From the computation of this paper, the local homology of the global fixed set contains the cohomology of $M$ with a degree shift. Spaces with big local homology seem hard to exist. Therefore, we have the following conjecture for general simply connected $M$. 
\begin{conj}
Let $M$ be a simply connected manifold that is not homeomorphic to $S^n$. If $\Homeo_0(M)$ acts nontrivially on $N$ of dimension $n+k<2n$, then the action has no global fixed points.
\end{conj}

When $M=S^n$, we obtain the following rigidity theorem.
\begin{thm}\label{main2}
For $\dim(N)=n+k<2n$, given a nontrivial homomorphism \[
\rho: \Homeo_0(S^n)\to \Homeo(N),
\]
there exists a $\ZZ$-homology manifold $K_1$ with boundary $K_0$ such that 
\[
N\cong  S^n \times K_1/\sim,
\]
where $(m,x)\sim (n,y)$ if $x=y\in K_0$. The action of $\rho$ on $N$ is induced by the standard action on $S^n$ and the trivial action on $K_1$.
\end{thm}
The above construction is a higher dimensional generalization of the ``coning-off" construction (or the suspension action) we discussed in \cite{CM}. When $K_1=[0,1]$, the action $\rho$ is the suspension action on $S^{n+1}=\Sigma S^n$.

\para{Proof idea of the paper} The tool we use other than orbit classification theorem is the Borel-Moore homology. The key ingredient is the local Smith theory which is why our theorem is limited to manifolds $M$ with finite actions. We also need to modify the ``shrinking argument" that we use in \cite{CM}. 

\para{Notation} When we use $\ZZ$-homology manifolds, we omit $\ZZ$; but for other coefficients, we keep the coefficient symbols. 

\para{Acknowledgement} We thank Shmuel Weinberger, Kathryn Mann and Benson Farb for helpful discussion.

\section{The proof of Theorem \ref{main}}
Let $\rho: \Homeo_0(M)\to \Homeo(N)$ be a homomorphism. Denote by \[
F=\Fix(\rho(\Homeo_0(M)),
\]
the global fixed point set of the action $\rho$. We have the following two ingredients of proof:
\begin{itemize}
\item By the free order $p$ element $\tau$ in $\Homeo_0(M)$, we know that $F$ is a union of finitely many $\FF_p$-homology manifolds (of various dimensions) by the local Smith theory;
\item A ``shrinking property" near a global fixed that we obtained from $\rho$ gives a way to compute the local homology of $F$.
\end{itemize}
\vskip 0.3cm

\para{Admissible covers of $M$}
Before we state the orbit classification theorem, we first discuss admissible covers of $M$.
\begin{defn}
A regular cover $p: M'\to M$ is called an \emph{admissible cover} if there exists a homomorphism\[
L:\Homeo_0(M)\to \Homeo(M'),
\]such that $L(f)$ is a lift of $f$. 
\end{defn}
For example, Simply-connected manifolds have trivial admissible covers; the universal cover of the circle $S^1$ is not admissible; however, the universal cover of $S_g$ a genus $g>1$ surface is admissible; when $M$ is connected Lie group, then admissible covers are trivial. More about which covers are admissible is discussed in \cite[Section 2.1]{CM}. 
\vskip 0.3cm

We now state the following orbit classification theorem and flat bundle structure theorem from \cite[Theorem 1.1 and Theorem 1.3]{CM}, which is the starting point of the whole argument.
\begin{thm}[orbit classification theorem and flat bundle structure theorem]\label{OCT}
Let $M,N$ be two connected, closed manifolds and let $\rho:\Homeo_0(M)\to \Homeo(N)$ be a homomorphism when $\dim(N)<2\dim(M)$. Then every orbit is either a global fixed point or a continuous image of an admissible cover of $M$. Let $F$ denote the global fixed point set. There is a topological flat bundle (homeomorphic to $\widetilde{M}\times K/\pi_1(M)$ for a topological space $K$ with a $\pi_1(M)$-action)
\[
p: N-F\to M,
\]
such that the action on $N-F$ is a lift of the standard action on $M$.
\end{thm}

Under the extra assumption that admissible covers of $M$ are trivial, we have the following corollary.
\begin{cor}
With the assumptions as in Theorem \ref{OCT}, and that any admissible cover of $M$ is trivial, we have that $N-F\cong M\times K$ such that $\Homeo_0(M)$ acts  standardly on $M$ and trivially on $K$.
\end{cor}

Denote by 
\[
\mathcal{E}: M\times K\to N-F
\]
the embedding obtained from the action $\rho$. 

\vskip 0.3cm

Combining Theorem \ref{OCT} with the existence of $\tau\in \Homeo_0(M)$ such that $\tau$ is a free order $p$ element, we know that the fixed point set of  $\rho(\tau)$ is also $F$ since the action of $\rho(\tau)$ on $N-F\cong M\times K$ is free. With the help of $\rho(\tau)$, the topology of the fixed point set $F$ is controlled by the local Smith theory.

We refer the reader to Bredon's book \cite{Bredon} for homology manifolds; see \cite[Definition 9.1 on Page 329]{Bredon}. Morally, a homology manifold is a locally compact space whose local homology is the same as that of an actual manifold but it may not have the local Euclidean structure. We first state the following local Smith theory \cite[Theorem 20.1 on Page 409]{Bredon}.
\begin{thm}[Local Smith theory]\label{LST}
Let $\phi\in\Homeo(N)$ be an order $p$ element. Then $\Fix(\phi)$ is a disjoint union of finitely many $\FF_p$-homology manifolds (of possibly various dimensions).
\end{thm}
By Theorem \ref{LST}, we know that $F=\Fix(\rho(\tau))$ is a union of finitely many $\FF_p$-homology manifolds. To compute the local homology of $F$, we need the following ``shrinking property". The ``shrinking property" is also used in \cite[Claim 5.3]{CM} to solve the extension problem of Ghys. 
\begin{prop}[Shrinking property]\label{shrinking}
For $m\in M$, let $\{k_n\in K\}$ be a sequence such that $\mathcal{E}(m,k_n)\in N-F$  converges to a global fixed point $P$ of $\rho$. Then the sequence of compact spaces $\mathcal{E}(M\times  \{k_n\})$ converges to $P$; i.e., for any neighborhood $U$ of $P$, there exist $N$ such that for $n>N$, we have that $\mathcal{E}(M\times \{k_n\})\subset U$.
\end{prop}
\begin{proof}
We prove the statement by contradiction. If not, then there exists a sequence $\{m_n\in M\}$ such that $\{\mathcal{E}(m_n,k_n)\}$ has a subsequence that does not converge to $P$. Passing to a further subsequence, assume that $m_n\to m'$, which is possible since $M$ is compact. Let $\{f_n\in\Homeo_0(M)\}$ be a sequence of elements such that 
\[
f_n\to f\text{ and }f_n(m)=m_n.\] 
Since the action $\rho$ is continuous (any homomorphism $\rho$ is automatically continuous by \cite{MannCont}) and 
\[
\mathcal{E}(m,k_n)\to P,
\]
we know that 
\[
\mathcal{E}(m_n,k_n)=\rho(f_n)(\mathcal{E}(m,k_n))\to \rho(f)(P)=P.\]
This clearly contradicts the assumption that $\{\mathcal{E}(m_n,k_n)\}$ does not converge to $P$.
\end{proof}
We now start the investigation of the local homology of $F$.
\begin{lem}\label{homology}
We have the following two statements:
\begin{enumerate}
\item $M$ is an $\FF_p$-homology sphere (a topological manifold with the same $\FF_p$-homology as a sphere).
\item $F$ is a disjoint union of $k-1$-dimensional $\FF_p$-homology manifolds (local homology, maybe not be a topological manifold).
\end{enumerate}
\end{lem}

\begin{proof}
Near a fixed point $P\in F$, we take a decreasing sequence of open neighborhoods $\{U_m\}$ such that $U_m$ is homeomorphic to the Euclidean space and $\bigcap_m U_m=\{x\}$. For the following computation, we use Borel-Moore homology; see e.g., \cite[Chapter 5]{Bredon}.

Denote by $A_m:=U_m\cap F$. By \cite[Page 293]{Bredon}, we have the following equation about the stalk of the local homology of $F$:
\begin{equation}\label{local}
\mathscr{H}_*(F;\FF_p)_x=\lim_{\longrightarrow} H_*(A_m;\FF_p),
\end{equation}
where the limit is taken over the induced map on homology of embeddings $A_{m+1}\hookrightarrow A_m$.  If $F$ is a homology manifold then $\mathscr{H}_*(F;\FF_p)_x=0$ for $*\neq \dim(F)$ at $x$.

We have the following universal coefficient theorem of Borel-Moore homology \cite[Page 292]{Bredon},
\[
0\to \text{Ext}(H^{p+1}_c(A_m;\FF_p);\FF_p)\to H_p(A_m;\FF_p)\to \text{Hom}(H^p_c(A_m;\FF_p);\FF_p)\to 0.
\]

By Poincar\'e duality \cite[Theorem 9.3 on Page 330]{Bredon}, we have the following 
\begin{equation}\label{PD}
H^{p}_c(A_m;\FF_p)\cong H_{k+n-p}^c(U_m,U_m-A_m;\FF_p) \cong H_{k+n-p-1}^{c(U_m)|_{U_m- A_m}}(U_m- A_m;\FF_p)\end{equation}
for $p\neq k+n-1$ (because $H^c_{*}(U_m;\FF_p)=0$ for $*\neq n+k$). However, by Theorem \ref{OCT}, there is a projection of flat bundle
\[
p: U_m-A_m\to M,
\]
and by Proposition \ref{shrinking}, there is an embedding for any $U_m$:
\[
M\xrightarrow{e_m} U_m-A_m
\]
such that $p\circ e_m=id$. 
Therefore we have the following embedding of homology group:
\[
e_{m*}: H_*^c(M;\FF_p)\hookrightarrow H_*^c(U_m-A_m;\FF_p),
\]
so that $e_{{m+1}*}$ factors through $e_{m*}$. In summary, we have the following embedding (a combination of \eqref{local} and \eqref{PD})
\[
H_{k+n-p-1}^c(M;\FF_p)\hookrightarrow \lim_{\longleftarrow} H_{k+n-p-1}^{c(U_m)|_{U_m- A_m}}(U_m-A_m;\FF_p)\cong \lim_{\longleftarrow} H^{p}_c(A_m;\FF_p).\]

Since $F$ is a local $\FF_p$-homology manifold, we know that $H_*(F;\FF_p)_x$ is nontrivial only when $*=\dim(F)$ at $x$. Since $M$ is a closed $n$-dimensional manifold, we know that $H_n(M;\FF_p)\cong \FF_p$. Therefore ${\mathscr H}_{k-1}(F;\FF_p)_x$ has an embedded copy of $\FF_p$, which implies that each connected component of $F$ is a $(k-1)$-dimensional $\FF_p$-homology manifold. 

Therefore, we know also that $H^*(M;\FF_p)$ is nontrivial only when $*=0$ or $n$.
\end{proof}
Theorem \ref{main} is an easy corollary of Lemma \ref{homology} and the proof is left to the reader.

\section{$\Homeo_0(S^n)$ acts on $N$}
In this section, we prove Theorem \ref{main2}. Before that, we show that the space described in Theorem \ref{main2} is still a homology manifold.
\begin{prop}
Let $K_1$ be a $\ZZ$-homology manifold with boundary $K_0$. Then
\[
N= S^n \times K_1/\sim,
\]
where $(m,x)\sim (n,y)$ if $x=y\in K_0$, is a homology manifold.
\end{prop}
\begin{proof}
The key is that we can redefine $N$ as a quotient of shrinking map. Let $\phi: S^n\times \partial K_1\to \partial D^{n+1}\times K_0$ be the identity map. Then
\[
N=S^n\times K_1 \cup_\phi D^{n+1}\times K_0/\sim,
\]
where $(x,t)\sim (y,t)$ for $x,y \in D^{n+1}$ and $t\in K_0$. We know that $S^n\times K_1 \cup_\phi D^{n+1}\times K_0$ is a homology manifold because it is a union of two homology manifolds glued along a common boundary. We know that $N$ is a homology manifold because of Wilder's monotone mapping theorem \cite[Theorem 16.33 on Page 389]{Bredon} and that $N$ is a quotient of a homology manifold by acyclic sets.
\end{proof}

For $\dim(N)=k+n<2n$, let \[
\rho:\Homeo_0(S^n)\to \Homeo(N)
\]
be a nontrivial homomorphism. 

Since $\dim(N)<2n$, every orbit is either a global fixed point or a continuous image of $S^n$ by Theorem \ref{OCT}. Therefore we have the following decomposition
\[
N\cong F\cup \mathcal{E}(S^n\times K)
\]
where $\mathcal{E}: S^n\times K\to N-F$ is the embedding obtained from the trivial flat bundle structure of Theorem \ref{OCT}.
By \cite[Theorem 16.11 on Page 378]{Bredon}, we know that $K$ is a homology manifold.
Firstly, we have the following description of $F$.
\begin{lem}\label{F2}
The fixed point set $F$ is a finite union of closed $k-1$-dimensional $\mathbb{F}_2$-homology manifold.
\end{lem}
\begin{proof}
Let $A=(\ZZ/2)^{n}\subset SO(n+1)$ be the subgroup consisting of diagonal matrices of entries $\pm 1$ with determinant $1$. Let $A_0\subset A$ be the subgroup of $A$ consisting of matrices whose last entry $1$. 

The action of $A$ on $S^n$ has no global fixed point. The action of $A_0$ on $S^n$ has two fixed points $(0,...,0,\pm 1)\in S^n$. Since the fixed point of $\rho(A)$ on $N$ is $F$, we know that $F$ is a union of finitely many $\FF_2$-homology manifolds.

Denote by $C$ the global fixed point set of $\rho(A_0)$. Since $C$ is an $\FF_2$-homology manifold and that $C\cap N-F$ is a homology manifold of dimension $k$, we know that $\dim(C)=k$ as an $\FF_2$-homology manifold. Since $F$ divides $C$ into two connected component. We know that $F$ has dimension $k-1$.
\end{proof}

We have the following stronger statement about $F$.
\begin{lem}
$F$ is a $k-1$ dimensional homology manifold and $F$ is the boundary of $\mathcal{E}(m\times K)$ for any $m\in M$. 
\end{lem}
\begin{proof}
Again, we use the subgroup $A,A_0$ as in the proof of Lemma \ref{F2}. Again denote by $C$ the global fixed point set of $\rho(A_0)$.

Define $B=\text{diag}(0,...,0,+1,-1)\in SO(n+1)$, which is an involution in $\Homeo_0(S^n)$. The action $\rho(B)$ restricts to an $\FF_2$ action on $C$. We now use the Mayer–Vietoris sequence and local Smith theory to compute local homology of $C$ at points of $F$.

Since the $\FF_2$-dimension of $F$ is $k-1$. A local version of \cite[Theorem 19.11 on Page 146]{Bredon} (See also \cite[Exercise 31 on Page 415]{Bredon}) gives us that for $x\in F$,
\[
H_*(C/\rho(B);\ZZ)_x=0 \]
for any $*$. Denote  $N:=(0,...,+1)\in S^{n}$ and $S:=(0,...,-1)\in S^{n}$.

By a local Mayer–Vietoris sequence, we have the following exact sequence
\[
...\to H_*(F;\ZZ)_x\to H_*(F\cup \mathcal{E}(\{N\}\times K);\ZZ)_x\oplus H_*(F\cup \mathcal{E}(\{S\}\times K);\ZZ)_x \to H_*(C;\ZZ)_x \to ...
\]
Therefore, we obtain that \[
H_*(F;\ZZ)_x\cong  H_{*+1}(C;\ZZ)_x.\]This implies that $F$ is a $k-1$ dimensional homology manifold since $C$ is a $k$ dimensional homology manifold.

The fact that $F$ is the boundary of $\mathcal{E}(\{N\}\times K)$ follows from the fact that $F$ cuts $C$ into two pieces $\mathcal{E}(\{N\}\times K)$ and $\mathcal{E}(\{S\}\times K)$.
\end{proof}
We now prove Theorem \ref{main2}.
\begin{proof}
Denote by $\overline{K}$ the manifold with boundary $\mathcal{E}(\{N\}\times K)\cup F$. We have a continuous map:
\[
\pi: S^n\times \overline{K}\to N
\]
by $\pi(s,k)=\mathcal{E}(s,k)$ for $k\in \text{Int}(\overline{K})$ and $\pi(s,k)=k$ if $k\in \partial \overline{K}=F$.

The map $\pi$ is continuous at points of $S^n\times K$ obviously and $\pi$ is also continuous at points of $S^n\times F$ because of Proposition \ref{shrinking} (the ``shrinking property"). Therefore $N$ is as described as in Theorem \ref{main2}.
\end{proof}

    	\bibliography{citing}{}
	\bibliographystyle{alpha}

\end{document}